\documentclass{amsart}
\usepackage{amsfonts,amsmath}
\usepackage{amssymb,latexsym,amsthm}
\usepackage[latin1]{inputenc}

\theoremstyle{plain}
\newtheorem{thm}{Theorem}[section]
\newtheorem{lemma}[thm]{Lemma}
\newtheorem{prop}[thm]{Proposition}
\newtheorem{cor}[thm]{Corollary}

\theoremstyle{remark}

 \def\a{{\alpha}}
 \def\b{{\beta}}
 
 \def\k{{\kappa}}
 
 \def\l{{\lambda}}

 \def\la{{\langle}}
 \def\ra{{\rangle}}

 \def\cb{{\mathbf c}}

 \def\xb{{\mathbf x}}

 \def\Kb{{\mathbf K}}

 \def\CO{{\mathcal O}}
 \def\CP{{\mathcal P}}
 \def\CM{{\mathcal M}}
 
 \def\CV{{\mathcal V}}

 \def\KK{{\mathbb K}}
 \def\NN{{\mathbb N}}
 \def\PP{{\mathbb P}}
 \def\QQ{{\mathbb Q}}
 \def\RR{{\mathbb R}}

  \newcommand{\tr}{{\mathsf {tr}}}
        \def\sspan{\operatorname{span}}

\title[Orthogonal polynomials for measures with mass points]
{Orthogonal polynomials in several variables for measures with mass points}

\author[A. M. Delgado]{Antonia M. Delgado}
\address[A. M. Delgado]{Departamento de Matem\'atica Aplicada,
Universidad de Granada, Spain.}

\author[L. Fern\'{a}ndez]{Lidia Fern\'{a}ndez}
\address[L. Fern\'{a}ndez]{Departamento de Matem\'atica Aplicada,
and Instituto Carlos I de F\'{\i}sica Te\'orica y Computacional,
Universidad de Granada, Spain}

\author[T. E. P\'erez]{Teresa E. P\'erez}
\address[T. E. P\'erez]{Departamento de Matem\'atica Aplicada,
and Instituto Carlos I de F\'{\i}sica Te\'orica y Computacional,
Universidad de Granada, Spain}

\author[M. A. Pi\~{n}ar]{Miguel A. Pi\~{n}ar$^{*}$}
\address[M. A. Pi\~{n}ar]{Departamento de Matem\'atica Aplicada,
and Instituto Carlos I de F\'{\i}sica Te\'orica y Computacional,
Universidad de Granada, Spain}

\author[Y. Xu]{Yuan Xu}
\address[Y. Xu]{Department of Mathematics,
Universty of Oregon, USA}

\thanks{$^{*}$ Corresponding author. E--mail: mpinar@ugr.es}

\thanks{Partially supported by Ministerio de Ciencia y Tecnolog\'{\i}a
        (MCYT) of Spain and by the European Regional Development Fund
        (ERDF) through the grant MTM 2008--06689--C02--02, and Junta de
        Andaluc\'{\i}a, Grupo de Investigaci\'on FQM 0229.}

\begin{document}



\begin{abstract}
Let $d\nu$ be a measure in $\mathbb{R}^d$ obtained from adding a set of mass points to
another measure $d\mu$. Orthogonal polynomials in several variables associated with $d\nu$
can be explicitly expressed in terms of orthogonal polynomials associated with $d\mu$, so are
the reproducing kernels associated with these polynomials. The explicit formulas that are
obtained are further specialized in the case of Jacobi measure on the simplex, with mass
points added on the vertices, which are then used to study the asymptotics kernel
functions for $d\nu$.
\end{abstract}

\maketitle

\section{Introduction}
\setcounter{equation}{0}

Let $d\mu$ be a measure on $\RR^d$ with all finite moments and we assume that $d \mu$
is positive definite in the sense that $\int_{\RR^d} p^2(x)  d\mu >0$ for every $p \in \Pi^d$,
$p \ne 0$, where $\Pi^d$ denotes the space of real polynomials in $d$--variables. Let
$\langle \cdot,\cdot \rangle_\mu$ denote the inner product defined by
\begin{equation} \label{ipd1}
     \langle p, q \rangle_\mu : = \int_{\RR^d} p(x)q(x) d\mu(x), \qquad p, q \in \Pi^d.
\end{equation}
Then orthogonal polynomials of several variables with respect to $\langle p,q\rangle_\mu$
exist. Let $N \ge 1$ be a positive integer and let $\xi_1, \xi_2, \ldots, \xi_N$ be distinct
points in $\RR^d$. Let $\Lambda$ be a positive definite matrix of size $N \times N$.
We define a new inner product $\langle \cdot,\cdot \rangle_\nu$ by
\begin{equation} \label{ipd2}
 \langle p, q \rangle_\nu : = \langle p, q \rangle_\mu
      + (p(\xi_1), p(\xi_2), \ldots, p(\xi_N)) \Lambda (q(\xi_1), q(\xi_2), \ldots, q(\xi_N))^\tr,
\end{equation}
where the superscript $\tr$ indicates the transpose, which can be defined via an integral
as in \eqref{ipd1} against a measure $d\nu$ that is obtained from adding $N$ mass points
to $d\mu$. A typical example is when $\Lambda$ is a diagonal matrix with positive entries.
The purpose of this paper is to study orthogonal polynomials with respect to the new
inner product $\langle \cdot,\cdot \rangle_\nu$.

In the case of one--variable, the first study of orthogonal polynomials for measures with
mass points was carried out, as far as we know, by Uvarov (\cite{U}), who gave a short
discussion on the case of adding a finite set of mass points to a measure and showed
how to express the orthogonal polynomials with respect to the new measure in terms
of those with respect to the old one. The problem was later revitalized by A. M. Krall
(\cite{K1}), who considered orthogonal polynomials for measures obtained by adding
mass points at the end of the interval on which a continuous measure lives. The case
of Jacobi measure with additional mass points at the end of $[-1,1]$ was studied in
\cite{Koor}, where explicit formulas of orthogonal polynomials were constructed. For
Jacobi weight with multiple mass points, it is possible to study asymptotic properties
of orthogonal polynomials  \cite{GPRV}. In the case of several variables, however,
only the case of $N =1$ has been studied \cite{FPPX}.

Our main results contain explicit formulas that express orthogonal polynomials and
reproducing kernels with respect to $\la\cdot, \cdot \ra_\nu$ in terms of those with
respect to $\la\cdot,\cdot\ra_\mu$. These results are stated and proved in Section 2.
As an example, we consider the case of Jacobi weight function on the simplex in $\RR^d$,
with mass points added at the vertices, for which our formulas can be further specified
and expressed in terms of the classical Jacobi polynomials. The result is then used to
study the asymptotic expansion of the Christoffel functions with respect to
$\la\cdot, \cdot\ra_\nu$.

\section{Orthogonal polynomials for measures with mass points}
\setcounter{equation}{0}

We start with a short subsection on necessary definitions, for which we follow essentially
\cite{DX}, and prove our main results in the second subsection.

\subsection{Preliminary}
Through this paper, we will use the standard multi--index notation. Let $\mathbb{N}_0$
denote the set of nonnegative integers. For a multi--index $\alpha=(\alpha_1,\dots,\alpha_d)
\in\mathbb{N}_0^d$ and $x=(x_1,\dots,x_d) \in \RR^d$, a monomial in $d$ variables is
defined as $x^{\alpha}=x_1^{\alpha_1}\cdots x_d^{\alpha_d}$. The integer
$|\alpha|=\alpha_1+ \dots+\alpha_d$ is called the \emph{total degree} of $x^{\alpha}$.
We denote by $\CP_n^d$ the space of homogeneous polynomials of degree $n$ in
$d$--variables, $\CP_n^d := \sspan\{x^\a: |\a| =n\}$, and denote by $\Pi_n^d$
the space of polynomials of total degree at most $n$. The collection of all polynomials
in $d$--variables is $\Pi^d$. It is well known that
$$
\dim \Pi_n^d = \binom{n+d}{n} \quad \hbox{\rm and}\quad \dim
\mathcal{P}_n^d = \binom{n+d-1}{n}:=r_n^d.
$$

Let $\langle \cdot,\cdot \rangle_\mu$ be the inner product defined in \eqref{ipd1}.
A polynomial $p \in \Pi_n^d$ is {\it orthogonal} with respect to \eqref{ipd1} if
$$
    \langle p,\, q\rangle_\mu = 0, \qquad \forall q\in \Pi_{n-1}^d.
$$
Our assumption that $d\mu$ is positive definite implies that orthogonal polynomials
with respect to $\langle \cdot,\cdot \rangle_\mu$ exist. Let us denote by $\mathcal{V}_n^d$
the space of orthogonal polynomials of total degree $n$. It follows that
$\dim \mathcal{V}_n^d =  r_n^d$. Let $\{P_\alpha^n\}_{|\alpha|=n}$ denote a basis of
$\mathcal{V}_n^d$. It is often convenient to use vector notations introduced in \cite{Ko1}
and \cite{X93}. Let $\{\alpha_1, \a_2,\ldots, \a_{r_n^d}\}$ be an enumeration of the set
$\{\alpha \in \NN_0^d: |\alpha| =n\}$ according to a fixed monomial order, say the
lexicographical order or the reversed lexicographical order. Then the basis
$\{P_\alpha^n\}_{|\alpha|=n}$ can be written as
$$
 \mathbb{P}_n= \left \{P^n_{\alpha_1}, P^n_{\alpha_2}, \ldots,  P^n_{\alpha_{r_n^d}} \right \}.
$$
We will treat $\PP_n$ both as a set of functions and as a {\it column} vector of functions.
As column vectors, the orthogonality of $\{P_{\a_j}^n\}$ can be expressed as
$$
  \la \PP_n, \PP_m^\tr \ra_\mu = \int_{\RR^d} \PP_n(x) \PP_m^\tr(x) d\mu = \begin{cases}
        0, & \hbox{if $n \ne m$}, \\ H_n, & \hbox{if $n = m$}, \end{cases}
$$
where the superscript denotes the transpose (so that $\PP^\tr$ is a row vector) and
$H_n$ is a matrix of size $r_n^d \times r_n^d$, necessarily symmetric, and in
fact a positive definite matrix by our assumption on $d\mu$. For convenience, we shall
call the system $\{\mathbb{P}_n\}_{n=0}^\infty = \{P_\alpha^n:  |\a| =n, n=0, 1,\ldots\}$ an
{\it orthogonal polynomial system} (OPS). If $H_n$ is the identity matrix, then
$\{P_\a^n: |\a|=n\}$ is an orthonormal basis for $\CV_n^d$ and the OPS is called an \emph{orthonormal} polynomial system.

Likewise, we can write $\xb^n : = \{x^\a: |\a| =n\} = \{x^{\a_1}, x^{\a_2}, \ldots, x^{\a_{r_n^d}} \}$
and regard it as a {\it column} vector. Since each element in $\PP_n$ is a polynomial of degree
$n$, it can be written as a sum of monomials, which, in vector notation, becomes
$$
\mathbb{P}_n = \sum_{j=0}^n G_{j,n}\,  \xb^j, \qquad \hbox{where} \quad G_{j,n} \in
  \mathcal{M}_{r_n^d\times r_j^d},
$$
in which $\CM_{p \times q}$ denotes the set of real matrices of size $p \times q$. In
particular, $G_{n,n}$ is a square matrix and it is necessarily invertible since $\PP_n$
is a basis of $\CV_n^d$. We call $G_{n,n}$ the leading coefficient of $\PP_n$.

With respect to $d\mu$, the reproducing kernel of $\CV_n^d$, denoted by $P_n(d\mu;x,y)$,
is defined by $\la P_n(d\mu; x, \cdot), p\ra_\mu = p(x)$, $p\in \CV_n^d$. In terms of
a basis $\PP_n$ of $\CV_n^d$, it satisifes
$$
    P_n(d\mu;x,y) =  \mathbb{P}^\tr_n(x)\, H^{-1}_n\, \mathbb{P}_n(y) \quad\hbox{with}
      \quad H_n = \la \PP_n, \PP_n\ra_\mu.
$$
Similarly, the reproducing kernel of $\Pi_n^d$, denoted by $K_n(d\mu;x,y)$, is defined by
$\la K_n(d\mu; x, \cdot), p\ra_\mu = p(x)$, $p\in \Pi_n^d$, and satisfies
$$
      K_n(d\mu;x,y) = \sum_{j=0}^n P_j(d\mu;x,y), \qquad n\ge 0.
$$
Since the definitions of $P_n(d\mu;x,y)$ and $K_n(d\mu;x,y)$  are independent of the
choice of a particular basis, (see \cite[Theorem 3.5.1]{DX}), it is often more
convenient to work with an orthonormal basis. The kernel $K_n(d\mu;x,y)$ plays an
important role in studying Fourier orthogonal expansions, as it is the kernel function of
the partial sum operator. The reciprocal of $K_n(d\mu;x,x)$ is called Christoffel function,
denoted by $\Lambda_n(x)$, and it satisfies
$$
\Lambda_n(x) : = \frac{1}{K_n(d\mu;x,x)} = \inf_{P(x) =1, P\in\Pi_n^d} \int_{\RR^d} |P(y)|^2 d\mu(y).
$$

\subsection{Main results}
Our goal is to study orthogonal polynomials with respect to the inner product
$\la \cdot,\cdot \ra_\nu$ defined in \eqref{ipd2}. Let us recall that $\Lambda$ is a given
positive definite matrix of order $N$ and $\{\xi_1, \xi_2, \ldots, \xi_N\}$ is a set of distinct
points in $\mathbb{R}^d$.  Introducing the notation
$$
 \mathbf{p}(\xi) = \left \{ p(\xi_1), p(\xi_2), \ldots, p(\xi_N) \right \},
$$
and regarding it also as a column vector, we can then rewrite the inner product
$\la \cdot,\cdot \ra_\nu$ in \eqref{ipd2} as
\begin{equation*}
(1.2') \hspace{1.5in} \hfill   \langle p, q \rangle_\nu = \langle p, q \rangle_\mu +\mathbf{p}(\xi)^\tr \, \Lambda \,\mathbf{q}(\xi),  \hspace{1.5in} \hfill
\end{equation*}
where $\langle \cdot, \cdot \rangle_\mu$ denotes the inner product defined in \eqref{ipd1}.
In the case that $\Lambda$ is a diagonal matrix, $\Lambda = \mathrm{diag}
\{\l_1,\ldots,\l_N\}$, the inner product $\la \cdot,\cdot \ra_\nu$ takes the form
\begin{equation} \label{sum-mass}
  \langle p, q \rangle_\nu = \langle p, q \rangle_\mu + \sum_{j=1}^N \l_j p(\xi_j) q(\xi_j).
\end{equation}

Our first result shows that orthogonal polynomials with respect to $\langle p, q \rangle_\nu$
can be derived in terms of those with respect to $\langle p, q \rangle_\mu$. The statement
and the proof of this result relies heavily on the vector--matrix notation. To facilitate the
study, we shall introduce several new notations.

Throughout this section, we shall fix $\PP_n$ as an orthonormal basis for $\CV_n^d$
associated with $d\mu$. We denote by  $\mathsf{P}_n(\xi)$ the matrix that has $\PP_n(\xi_i)$
as columns,
\begin{equation} \label{sP}
 \mathsf{P}_n(\xi):= \left(\mathbb{P}_n(\xi_1) | \mathbb{P}_n(\xi_2) | \ldots
   | \mathbb{P}_n(\xi_N)   \right) \in \mathcal{M}_{r_n^d \times N},
\end{equation}
denote by $\Kb_{n-1}$ the matrix whose entries are $K_{n-1}(d\mu;\xi_i,\xi_j)$,
\begin{equation} \label{cK}
\mathbf{K}_{n-1} := \big(K_{n-1}(d\mu;\xi_i,\xi_j) \big)_{i,j=1}^N \in \mathcal{M}_{N\times N},
\end{equation}
and, finally, denote by $\mathbb{K}_{n-1}(\xi,x)$ the vector of functions
\begin{equation} \label{sK}
  {\mathbb{K}_{n-1}(\xi,x)} = \left\{K _{n-1}(d\mu; \xi_1,x),  K_{n-1}(d\mu;\xi_2,x),
\ldots, K_{n-1}(d\mu;\xi_N,x)\right \},
\end{equation}
which we again regard as a column vector.

From the fact that $K_n(d\mu; x,y) - K_{n-1}(d\mu;x,y) = P_n(d\mu;x,y)$, we have immediately
the following relations,
\begin{align}
   \mathsf{P}_n^\tr(\xi) \PP_n (x) & = \KK_n(\xi,x) -   \KK_{n-1}(\xi,x),  \label{P-K1}\\
   \mathsf{P}_n^\tr(\xi) \mathsf{P}_n (\xi) & = \Kb_n - \Kb_{n-1},  \label{P-K2}
\end{align}
which will be used below. Let $I_N$ denote the identity matrix of order $N$.

\begin{lemma}
The matrix $I_N +  \Lambda\,{\mathbf{K}_{n-1}}$ is invertible.
\end{lemma}

\begin{proof}
First we show that the matrix $\Kb_{n-1}$ is positive definite, By the definition of
$K_n(d\mu;\cdot,\cdot)$, for every $\cb \in \RR^N$, $\cb \ne 0$, we have
$$
  \cb^\tr \Kb_{n-1} \cb = \sum_{|\a| \le n-1}  \sum_{i,j =1}^N c_i c_j P_\a(\xi_i) P_\a(\xi_j)
         =  \sum_{|\a| \le n-1} \Big| \sum_{j=0}^N c_j P_\a(\xi_j) \Big|^2 > 0,
$$
so that $\Kb_{n-1}$ is positive definite. The matrix $\Lambda$ is also positive definite,
by assumption, so that it is invertible. Since $\Lambda^{-1} (I_N + \Lambda \Kb_{n-1}) =
\Lambda^{-1} + \Kb_{n-1}$, we see that it is positive definite as well, hence invertible.
Consequently, $I_N + \Lambda \Kb_{n-1}$ is invertible.
\end{proof}

We are now ready to state and prove our first main result.

\begin{thm} \label{main-thm}
Define a polynomial system $\{\mathbb{Q}_n\}_{n\ge 0}$ by $\mathbb{Q}_0(x) := \mathbb{P}_0(x)$ and
\begin{equation} \label{ex-expl}
\mathbb{Q}_n(x) = \mathbb{P}_n(x) - {\mathsf{P}_n(\xi)}\,(I_N +
\Lambda\,{\mathbf{K}_{n-1}})^{-1}\,\Lambda \,
{\mathbb{K}_{n-1}(\xi,x)}, \qquad n\ge 1.
\end{equation}
Then $\{\mathbb{Q}_n\}_{n\ge 0}$ is a sequence of orthogonal polynomials
with respect to $\la \cdot,\cdot\ra_\nu$ defined in \eqref{ipd2}. Conversely, any
sequence of orthogonal polynomials with respect to \eqref{ipd2} can be expressed
as in \eqref{ex-expl}.
\end{thm}

\begin{proof}
Let us assume that $\{\QQ_n\}_{n\ge 0}$ is an OPS with respect to $\la \cdot, \cdot \ra_\nu$
and $\QQ_n$ has the same leading coefficient as $\PP_n$, which implies, in particular, that
$\QQ_0$ is a constant and $\QQ_0 = \PP_0$. We show that $\QQ_n$ satisfies \eqref{ex-expl}.
By the assumption, the components of $\mathbb{Q}_n-\mathbb{P}_n$ are elements in
$\Pi_{n-1}^d$ for $n\ge 1$. Since $\{\mathbb{P}_n\}_{n \ge 0}$ is a basis of $\Pi^d$, we
can express these components as linear combinations of orthogonal polynomials in
$\PP_0, \PP_1, \ldots, \PP_{n-1}$. In vector-matrix notation, this means that
$$
\mathbb{Q}_n(x) = \mathbb{P}_n(x) +  \sum_{j=0}^{n-1} M_j^n\, \mathbb{P}_j(x),
$$
where $M_j^n$ are matrices of size $r_n^d\times r_j^d$. These coefficient matrices can
be determined from the orthonormality of $\mathbb{P}_n$ and $\QQ_n$. Indeed, by the
orthogonality of $\QQ_n$, $\la \QQ_n, \PP_j \ra_\nu =0$  for $0 \le j \le n-1$, which shows,
by the definition of $\la \cdot,\cdot \ra_\nu$ and the fact that $\PP_j$ is orthonormal,
$$
    M_j^n = \langle \mathbb{Q}_n, \mathbb{P}^\tr_j \rangle_\mu =
 -  \mathsf{Q}_n(\xi)^\tr  \Lambda \mathsf{P}_n(\xi),
$$
where $\mathsf{P}_n(\xi)$ is defined as in \eqref{sP} and $\mathsf{Q}_n(\xi)=
\left\{\mathbb{Q}_n(\xi_1) | \mathbb{Q}_n(\xi_2) | \ldots | \mathbb{Q}_n(\xi_N)\right \}$ in
the analogous matrix with $\QQ_n(\xi_i)$ as its column vectors. Consequently, we obtain
\begin{align} \label{QnPj}
\mathbb{Q}_n(x) & = \mathbb{P}_n(x) -  \sum_{j=0}^{n-1} \mathsf{Q}_n(\xi)\,\Lambda
 \mathsf{P}^\tr_j(\xi) \, \mathbb{P}_j(x) \\
    & = \mathbb{P}_n(x) -  \mathsf{Q}_n(\xi)\,\Lambda \,\mathbb{K}_{n-1}(\xi,x). \notag
\end{align}
where the second equation follows from the relation \eqref{P-K1}, which leads to a
telescoping sum that sums up to $\mathbb{K}_{n-1}(\xi,x)$. Setting $x=\xi_i$, we obtain
$$
\mathbb{Q}_n(\xi_i) = \mathbb{P}_n(\xi_i) -  \mathsf{Q}_n(\xi)\,\Lambda \,\mathbb{K}_{n-1}(\xi,\xi_i), \quad 1 \le i \le N,
$$
which leads to, by the definition of $\Kb_{n-1}$ at (\ref{cK}), that
$$
   \mathsf{Q}_n(\xi) = \mathsf{P}_n(\xi) -  \mathsf{Q}_n(\xi)\,\Lambda \,\mathbf{K}_{n-1}.
$$
Solving for $\mathsf{Q}_n (\xi)$ from the above equation gives
\begin{equation}\label{qn(c)}
    \mathsf{Q}_n(\xi) = \mathsf{P}_n(\xi) (I_N + \,\Lambda \,\mathbf{K}_{n-1})^{-1}.
\end{equation}
Substituting this expression into \eqref{QnPj} establishes (\ref{ex-expl}).

Conversely, if we define polynomials $\mathbb{Q}_n$ by \eqref{ex-expl}, the above proof
shows that $\mathbb{Q}_n$ is orthogonal with respect to $\langle \cdot, \cdot \rangle_\nu$.
Since $\mathbb{Q}_n$ and $\mathbb{P}_n$ have the same leading coefficient, it is evident
that $\{\mathbb{Q}_n \}_{n \ge 0}$ is an OPS in $\Pi^d$.
\end{proof}

Let $\{\mathbb{Q}_n\}_{n\ge0}$ be an OPS with respect to \eqref{ipd2} as in
Theorem \ref{main-thm}. In general, $\QQ_n$ is not orthonormal. We denote, in the
rest of this section,
$$
H_n := \langle \mathbb{Q}_n, \mathbb{Q}^\tr_n \rangle_\nu.
$$
Then $H_n$ is a positive definite matrix. It turns out that both $H_n$ and $H_n^{-1}$ can
be expressed in terms of matrices that involve only $\{\PP_j\}_{j \ge 0}$.

\begin{prop} \label{prop.3.2}
For $n \ge 0$,
\begin{align}
   H_n & = I_{r_n^d} + \mathsf{P}_n(\xi) (I_N + \,\Lambda \,\mathbf{K}_{n-1})^{-1} \Lambda \mathsf{P}_n^\tr(\xi),  \label{directa} \\
  H^{-1}_n & = I_{r_n^d} - \mathsf{P}_n(\xi) (I_N + \,\Lambda \,\mathbf{K}_{n})^{-1} \Lambda \mathsf{P}_n^\tr(\xi). \label{inversa}
\end{align}
\end{prop}

\begin{proof}
Since $\PP_n$ is orthonormal, $\langle \mathbb{P}_n, \mathbb{P}^\tr_n \rangle_\mu
 = I_{r_n^d}$. From \eqref{ex-expl} and \eqref{qn(c)} we obtain
\begin{align*}
 H_n =& \langle \mathbb{Q}_n , \mathbb{Q}_n^\tr\rangle_\nu =
\langle \mathbb{Q}_n , \mathbb{P}_n^\tr\rangle_\nu = \langle
\mathbb{Q}_n , \mathbb{P}_n^\tr\rangle_\mu + \mathsf{Q}_n(\xi) \Lambda \mathsf{P}_n^\tr(\xi) \\
   =& I_{r_n^d} + \mathsf{P}_n(\xi) ( I_N + \,\Lambda \,\mathbf{K}_{n-1})^{-1}
   \Lambda \mathsf{P}_n^\tr(\xi),
\end{align*}
which proves (\ref{directa}).  In order to establish \eqref{inversa}, we need to verify that
$$
H_n ( I_{r_n^d} - \mathsf{P}_n(\xi) (I_N + \,\Lambda \,\mathbf{K}_{n})^{-1} \Lambda \mathsf{P}_n^\tr(\xi) ) =  I_{r_n^d},
$$
which, by \eqref{directa} and after simplification, reduces to the following equation,
\begin{align}\label{identity}
(I_N + \Lambda\,\mathbf{K}_{n-1})^{-1}\, & \Lambda\,   \mathsf{P}_n(\xi)^\tr \mathsf{P}_n(\xi)
    (I_N + \Lambda\,\mathbf{K}_{n})^{-1} \\
  & = (I_N + \Lambda\,\mathbf{K}_{n-1})^{-1} - (I_N +\Lambda\,\mathbf{K}_{n})^{-1}. \notag
\end{align}
Using \eqref{P-K2}, the above equation can be verified by a simple computation.
\end{proof}

Our next result gives explicit formulas for the reproducing kernels associated with
$\la \cdot,\cdot\ra_\nu$, which we denote by
\begin{equation*}
P_j(d\nu; x,y):= \mathbb{Q}^\tr_j(x)\, H^{-1}_j\, \mathbb{Q}_j(y) \quad \hbox{and} \quad
  K_n(d\nu; x,y):= \sum_{j=0}^n P_j(d\nu;x,y).
\end{equation*}

\begin{thm} \label{thm.3.3}
For $j\ge 0$,
\begin{align}
P_j(d\nu; x,y) =   P_j(d\mu; x,y)  & -
   \mathbb{K}_j^\tr(\xi,x) \,  (I_N + \,\Lambda \,\mathbf{K}_{j})^{-1}  \Lambda \, \mathbb{K}_j(\xi,y)  \label{Pk}\\
 & + \mathbb{K}_{j-1}^\tr(\xi,x) \,  (I_N + \,\Lambda \,\mathbf{K}_{j-1})^{-1}  \Lambda \, \mathbb{K}_{j-1}(\xi,y). \nonumber
\end{align}
Furthermore, for $n \ge 0$,
\begin{equation}\label{kernel}
K_n(d\nu; x,y)  = K_n(d\mu; x,y)  -
\mathbb{K}_n^\tr(\xi,x) \,  (I_N + \,\Lambda \,\mathbf{K}_{n})^{-1}  \Lambda \, \mathbb{K}_n(\xi,y).
\end{equation}
\end{thm}

\begin{proof}
Since $\Lambda^{-1}\,(I_N + \Lambda\,\mathbf{K}_{j-1}) = \Lambda^{-1}\,+\,\mathbf{K}_{j-1}$
is a symmetric matrix, so is $(I_N + \Lambda\,\mathbf{K}_{j-1})^{-1}\,\Lambda$. Using this
fact, it follows from \eqref{ex-expl} and \eqref{inversa} that
\begin{align*}
 \mathbb{Q}^\tr_j(x)\, H^{-1}_j  = &  \mathbb{P}^\tr_j(x) -
   \PP^\tr_j(x) \mathsf{P}_j(\xi)(I - \Lambda \Kb_j)^{-1} \Lambda \mathsf{P}^\tr_j(\xi) \\
  & - \mathbb{K}_{j-1}^\tr(\xi,x)  \,(I_N +\Lambda\,\mathbf{K}_{j-1})^{-1}\,\Lambda\,
     \mathsf{P}_j^\tr(\xi)  \\
   & -  \mathbb{K}_{j-1}^\tr(\xi,x)  \,(I_N + \Lambda\,\mathbf{K}_{j-1})^{-1}\Lambda\,
      \mathsf{P}^\tr_j(\xi)  \mathsf{P}_j(\xi) (I+\Lambda \Kb_j)^{-1} \mathsf{P}_j^\tr(\xi),
\end{align*}
which simplifies to, upon using \eqref{identity} and \eqref{P-K1},
\begin{align*}
 \mathbb{Q}^\tr_j(x)\, H^{-1}_j  = & \mathbb{P}_j^\tr(x)  -
   \PP_j^\tr(x)  \mathsf{P}_j(\xi)(I - \Lambda \Kb_j)^{-1} \Lambda \mathsf{P}_j^\tr(\xi)  \\
    & - \mathbb{K}_{j-1}^\tr(\xi,x) \,(I_N +
    \Lambda\,\mathbf{K}_{j-1})^{-1}\,\Lambda\, \mathsf{P}_j^\tr(\xi) \\
  =  & \mathbb{P}_j^\tr(x) - \mathbb{K}_{j}^\tr(\xi,x) \,(I_N +
\Lambda\,\mathbf{K}_{j})^{-1}\,\Lambda\, \mathsf{P}_j^\tr(\xi).
\end{align*}
Using again \eqref{ex-expl} and \eqref{P-K1}, we then obtain
\begin{align*}
 \mathbb{Q}^\tr_j(x)\,& H^{-1}_j \mathbb{Q}_j(y) =
   \mathbb{P}_j^\tr(x) \,\mathbb{P}_j(y)  \\
&- [\mathbb{K}_{j}^\tr(\xi,x) - \mathbb{K}_{j-1}^\tr(\xi,x)] \,(I_N +
\Lambda\,\mathbf{K}_{j-1})^{-1}\,\Lambda\, \mathbb{K}_{j-1}(\xi,y) \\
&-  \mathbb{K}_{j}^\tr(\xi,x) \,(I_N +
\Lambda\,\mathbf{K}_{j})^{-1}\,\Lambda\, [\mathbb{K}_{j}(\xi,y) - \mathbb{K}_{j-1}(\xi,y)] \\
&+ \mathbb{K}_{j}^\tr(\xi,x) \,(I_N +
\Lambda\,\mathbf{K}_{j})^{-1}\,\Lambda\,\mathsf{P}^\tr_j(\xi)  \mathsf{P}_j(\xi) (I_N + \Lambda\,\mathbf{K}_{j-1})^{-1} \Lambda \mathbb{K}_{j-1}(\xi,y),
\end{align*}
which simplifies to \eqref{Pk} upon using the identity \eqref{identity}.

Finally, summing over \eqref{Pk} for $j = 0, 1, \ldots, n$, we obtain \eqref{kernel}.
\end{proof}

The result in this section can be extended without much difficulty to mass points with derivative
values. To be more precise, let $\partial^\a = \partial_1^{\a_1} \cdots \partial_d^{\a_d}$, where
$\partial_i = \frac{\partial}{\partial x_i}$, and for $\a_i \in \NN_0^d$, $i = 1,2,\ldots, N$, define
\begin{equation} \label{lpd-Soblev}
 D_\a\mathbf{p}(\xi) := \left \{ \partial^{\a_1} p(\xi_1),  \partial^{\a_2}p(\xi_2), \ldots,
  \partial^{\a_N}p(\xi_N) \right \},
\end{equation}
and regard it also as a column vector. Instead of requiring $\xi_i \ne \xi_j$, we only assume
that $\xi_i \ne \xi_j$ when $\a_i = \a_j$. In other word, $\xi_i$ and $\xi_j$ can be the same
as long as $\a_i \ne \a_j$. We then consider the inner product defined by
\begin{equation} \label{soblev}
 \langle p, q \rangle_\nu = \langle p, q \rangle_\mu + D_\a\mathbf{p}(\xi)^\tr \, \Lambda \,
  D_\a \mathbf{q}(\xi).
\end{equation}
When $\a_i =0$ for all $i$, this is the inner product in \eqref{ipd2}. Other interesting cases
include, for example,
$$
   \langle p, q \rangle_\nu = \langle p, q \rangle_\mu +
   \sum_{j =0}^N \l_j p(\xi_j) q(\xi_j) + \sum_{j=0}^N \lambda_j' \nabla p(\xi_j) \cdot \nabla q(\xi_j).
$$

Our results in Theorem~\ref{main-thm}, Proposition~\ref{prop.3.2} and Theorem~\ref{thm.3.3}
still hold in this setting, but we need to replace $\mathsf{P}_n$ in \eqref{sP}, $\Kb_{n-1}$ in \eqref{cK},  and $\KK_{n-1}(\xi,x)$ in \eqref{sK}  by
\begin{align*}
\mathsf{P}_n^*(\xi):= & \left(\partial^{\a_1}\mathbb{P}_n(\xi_1) | \partial^{\a_2}\mathbb{P}_n(\xi_2)
   | \ldots  | \partial^{\a_d}\mathbb{P}_n(\xi_N)   \right) \in \mathcal{M}_{r_n^d \times N}, \\
\mathbf{K}_{n-1}^* := & \big(\partial^{\a_i}_{\{1\}} \partial^{\a_j}_{\{2\}}
   K_{n-1}^*(d\mu;\xi_i,\xi_j) \big)_{i,j=1}^N \in \mathcal{M}_{N\times N}, \\
  {\mathbb{K}_{n-1}^*(\xi,x)} = & \left\{\partial^{\a_1}_{\{1\}}K _{n-1}(d\mu; \xi_1,x),
  \partial^{\a_2}_{\{1\}} K_{n-1}(d\mu;\xi_2,x),
    \ldots, \partial^{\a_d}_{\{1\}} K_{n-1}(d\mu;\xi_N,x)\right \}
\end{align*}
respectively, where $\partial^\a_{\{1\}} K_n(u,v)$ means that the derivative is taken with respect
to $u$ variable.

\begin{thm}
The results in Theorem~\ref{main-thm} and Theorem~\ref{thm.3.3} hold for the inner
product defined in \eqref{lpd-Soblev} when $\mathsf{P_n}$, $\Kb_{n-1}$ and $\KK_{n-1}(\xi,x)$
are replaced by $\mathsf{P}^*_n$, $\Kb_{n-1}^*$ and $\KK_{n-1}^*(\xi,x)$, respectively.
\end{thm}

The proof follows as before almost verbatim with little additional difficulty.

\section{Orthogonal polynomials on the simplex}
\setcounter{equation}{0}

In this section we apply the general result in the previous section to orthogonal polynomials on
the simplex
$$
T^d : = \{x = ({x_1},\ldots,
   {x_d})\in \mathbb{R}^d: {x_i} \ge 0, 1-|x|_1 \ge0 \}
$$
in $\mathbb{R}^d$, where $|x|_1 = x_1 + \ldots + x_d$.

\subsection{Jacobi polynomials on the simplex}
We consider the Jacobi weight function
$$
W_{\kappa}(x) = x_1^{\kappa_1 -1/2} \cdots  x_d^{\kappa_d -1/2} (1 - |x|_1)^{\kappa_{d+1} -1/2}, \quad \kappa_i \ge 0,
$$
on the simplex, where
$$
w_{\kappa} = \frac{\Gamma(|\kappa|+ \frac{d+1}{2})}{\Gamma(\kappa_1+ \frac{1}{2}) \cdots \Gamma(\kappa_{d+1} + \frac{1}{2})}, \qquad
      |\kappa| : = \kappa_1 + \kappa_2 + \cdots + \kappa_{d+1},
$$
is the normalization constant of $W_{\kappa}$ such that $w_{\kappa}\,\int_{T^d} W_{\kappa}(x)\,dx=1$. Associated with $W_\k$, we consider the inner product on the simplex
\begin{equation}\label{simplex-ipg}
 \langle f, g \rangle = w_{\kappa} \, \int_{T^d} f(x) \,g(x) \, W_{\kappa}(x) \, dx,
\end{equation}
which plays the role of $\la \cdot, \cdot \ra_\mu$ when we deal with the settings of
the previous section. For $d=1$, $W_\k$ is the classical Jacobi weight function, which
has orthogonal polynomials $P_n^{(\k_1,\k_2)}(2t -1)$, where $P_n^{(a,b)}$ is the
classical Jacobi polynomial of degree $n$ that is orthogonal with respect to
$(1-t)^a (1+t)^b$ on $[-1,1]$ and normalized by $P_n^{(a,b)}(1) = \binom{n+a}{n}$.
We shall also denote the orthonormal Jacobi polynomials by $p_n^{(a,b)}(t)$. Evidently,
$p_n^{(a,b)}(t) = c_n P_n^{(a,b)}(t)$, where the constant $c_n$ is given by \cite[(4.3.3)]{Sz}.

To state an orthonormal basis for $\CV_n^d$ on the simplex, we follow \cite[p. 47]{DX}
and introduce the following notation. Associated with $x = (x_1, \ldots, x_d) \in \mathbb{R}^d$,
we define by $\mathbf{x}_j$ the truncation of $x$, namely
$$
\mathbf{x}_0 = 0, \quad \mathbf{x}_j = (x_1, \ldots, x_j), \quad 1 \le j \le d,
$$
and associated with $\alpha = (\alpha_1, \ldots, \alpha_{d}) \in \mathbb{N}^d_0$ and
$\kappa = (\kappa_1, \ldots, \kappa_{d+1}) \in \mathbb{R}^{d+1}$, we introduce, respectively,
$$
 \alpha^j := (\alpha_j, \ldots, \alpha_{d}), \quad 1 \le j \le d, \qquad
 \kappa^j := (\kappa_j, \ldots, \kappa_{d+1}), \quad 1 \le j \le d+1.
$$
Then, an orthonormal basis associated with \eqref{simplex-ipg} is given explicitly by
\begin{equation} \label{simplex-base}
 P_{\alpha}(W_\kappa; x) = h^{-1}_{\alpha}
\prod_{j=1}^d \left(\frac{1- |\mathbf{x}_{j}|_1}{1- |\mathbf{x}_{j-1}|_1} \right)^{ |\alpha^{j+1}|}
 p_{\alpha_j}^{(a_j,b_j)}\left(\frac{2 x_j}{1- |\mathbf{x}_{j-1}|_1} - 1\right),
\end{equation}
where the parameters $a_j$ and $b_j$ are given by
$$
a_j = 2 |\alpha^{j+1}| + |\kappa^{j+1}| + \frac{d-j-1}{2}, \qquad b_j = \kappa_j - \frac{1}{2},$$
and $h_{\alpha}$ is the normalizing constant given by
$$
h^{2}_{\alpha} =  \frac{(|\kappa|+ \frac{d+1}{2})_{2 |\alpha|}}
{\prod_{j=1}^d (2|\alpha^{j+1}|+|\kappa^j|+ \frac{d-j+2}{2})_{2\alpha_j}},$$
in which $(a)_k := a(a+1) \ldots (a+ k-1)$ denotes the shifted factorial.

In this case, we also have a compact formula for the reproducing kernels, given in
terms of the Gegenbauer polynomials $C^{\lambda}_{n}$, which are orthogonal with
respect to the weight function $(1-t^2)^{\lambda-1/2}$ and normalized by
$C_n^\l(1) = \binom{n+2\l-1}{n}$. The formula, first derived in \cite[Theorem 2.3]{X98},
is given by
\begin{align} \label{KnSimplex}
 K_n(W_{\kappa}; x, y)= & \frac{1}{2^{d+1}} \, \int_{[-1,1]^{d+1}} C^{\lambda}_{2n}(\sqrt{x_1\,y_1}\,t_1 + \cdots +  \sqrt{x_{d+1}\, y_{d+1}}\, t_{d+1}) \\
   &\quad \times \prod_{j=1}^{d+1} c_{\kappa_j}\, (1-t_j^2)^{\kappa_j-1}\, dt, \notag
\end{align}
where $x_{d+1} = 1-|x|_1$, $y_{d+1} = 1-|y|_1$,  $\lambda:= |\kappa|+\frac{d+1}{2}$, and $c_{\kappa_j} = \int_{-1}^1 (1-t_j^2)^{\kappa_j-1}\, dt_i$.

Let us denote the standard Euclidean basis of $\RR^d$ by $\{e_1, \ldots, e_d\}$, where
$e_i = (0,\ldots,0,1,0\ldots, 0)$ with the single 1 in the $i$-th position. Furthermore, we
set $e_{d+1} = (0,\ldots, 0) \in \RR^d$. Then $\{e_1,e_2,\ldots, e_{d+1}\}$ is the set of
vertices of $T^d$.

\begin{prop}
Let $\lambda= |\kappa|+\frac{d+1}{2}$. For $1\le i \le d+1$, we have
\begin{equation}\label{kernel-s}
K_n(W_{\kappa}; x, e_i) = \frac{1}{2^{d+1}} \, \frac{(\lambda)_n}{(\kappa_i+1/2)_n} \,
   P^{(\lambda-\kappa_i-1/2,\kappa_i-1/2)}_{n}(2\, x_i-1).
\end{equation}
In particular, we have
\begin{align}
K_n(W_{\kappa}; e_i, e_i) &= \frac{1}{2^{d+1}} \frac{(\lambda)_n}{n!}
  \frac{(\lambda-\kappa_i+1/2)_n}{(\kappa_i+1/2)_n}, \quad 1\le i \le d+1. \label{kernel-i-i}  \\
K_n(W_{\kappa}; e_i, e_j) &= \frac{(-1)^n}{2^{d+1}} \frac{(\lambda)_n}{n!}, \quad 1\le i,j \le d+1, \label{kernel-k-i}
\end{align}
\end{prop}

\begin{proof}
Since $C^{\lambda}_{2n}$ is an even function, for  $1\le i \le d +1 $ we deduce from
\eqref{KnSimplex} that
\begin{align*}
K_n(W_{\kappa}; x, e_i) =& \frac{c_{\kappa_i}}{2^{d+1}} \, \int_{-1}^1
   C^{\lambda}_{2n}(\sqrt{x_i}\,t_i) \, (1-t_i^2)^{\kappa_i-1}\, dt_i\\
=& \frac{c_{\kappa_i}}{2^{d+1}} \, \int_{-1}^1 C^{\lambda}_{2n}(\sqrt{x_i}\,t_i) \, (1-t_i)^{\kappa_i-1}\, (1+t_i)^{\kappa_i}\,dt_i\\
=& \frac{1}{2^{d+1}} \, V^{(\kappa_i)} \, C^{\lambda}_{2n}(\sqrt{x_i}) \\
=& \frac{1}{2^{d+1}} \, \frac{(\lambda)_n}{(\kappa_i+1/2)_n} \, P^{(\lambda-\kappa_i-1/2,\kappa_i-1/2)}_{n}(2\, x_i-1),
\end{align*}
where $V^{(\kappa_i)}$ is the operator defined in \cite[Definition 1.5.1, p. 24]{DX} and the last equality comes from \cite[Proposition 1.5.6, p. 27]{DX}.  In particular, setting $x = e_j$ in
\eqref{kernel-s} shows that
$$
K_n(W_{\kappa}; e_i, e_i) = \frac{1}{2^{d+1}} \, \frac{(\lambda)_n}{(\kappa_i+1/2)_n} \,
 P^{(\lambda-\kappa_i-1/2,\kappa_i-1/2)}_{n}(1),
$$
and, for $i\neq j$,
$$
K_n(W_{\kappa}; e_j, e_i) = \frac{1}{2^{d+1}} \, \frac{(\lambda)_n}{(\kappa_i+1/2)_n} \,
  P^{(\lambda-\kappa_i-1/2,\kappa_i-1/2)}_{n}(-1),
$$
from which \eqref{kernel-i-i} and \eqref{kernel-k-i} follow from \cite[(4.1.1) and (4.1.4)]{Sz}).
\end{proof}

\subsection{Orthogonal polynomials on the simplex with mass points}
We consider orthogonal polynomials on the simplex for the Jacobi measure with additional
mass at each of the vertices of the simplex. In order to preserve symmetry, we shall limit
ourself to the situation that every vertex has the same weight $M>0$. In other words, we
consider the inner product
\begin{equation} \label{simplex-ip}
 \langle f, g \rangle_{\nu} =  w_{\kappa} \, \int_{T^d} f(x) \, g(x) \, W_{\kappa}(x) \, dx +
M \, \sum_{i=1}^{d+1} f(e_i) \, g(e_i), \qquad M > 0.
\end{equation}
In the language of the inner product \eqref{ipd2}, we assume that $\Lambda$ is a diagonal
matrix $\Lambda = M\, I_{d+1}$ and the inner product take the form of \eqref{sum-mass}.

We will further limit ourself to the case that $\kappa_1 = \kappa_2 = \cdots =
\kappa_{d+1} = \varsigma \ge 0$. Under this assumption,
$$
     \lambda = (d+1) (\varsigma + 1/2).
$$
We further denote
\begin{align*}
A_n &:=  K_n(W_{\kappa};e_i,e_i) = \frac{1}{2^{d+1}} \frac{(\lambda)_n}{n!} \frac{(\lambda-\varsigma + 1/2)_n}{(\varsigma + 1/2)_n},\\
B_n &:= K_n(W_{\kappa};e_j,e_i) = \frac{(-1)^n}{2^{d+1}} \frac{(\lambda)_n}{n!}, \qquad j\neq i.
\end{align*}
As a result, we see that the matrix $\Kb_n$ defined in \eqref{cK} is given by
\begin{align*}
{\mathbf{K}_{n}} = \, & (K_{n}(W_{\kappa};e_i,e_j))_{i,j=1}^{d+1} = \begin{pmatrix}
{A_n} & {B_n} & \cdots & {B_n}\\
{B_n} & {A_n} & \cdots & {B_n}\\
\vdots & \vdots & \ddots & \vdots\\
{B_n} & {B_n} & \cdots & {A_n}
\end{pmatrix} \\
= \, & ({A_n}-{B_n})I_{d+1} +
{B_n} \begin{pmatrix} 1 & \cdots & 1\\\vdots &  & \vdots\\
1 & \cdots & 1\end{pmatrix}\\
= \, & ({A_n}-{B_n})I_{d+1} +
{B_n} \begin{pmatrix} 1 \\  \vdots \\1\end{pmatrix} (1,\ldots,1).
\end{align*}
This shows that $\mathbf{K}_n$ is a rank one perturbation of the identity matrix and,
consequently, the inverse of the matrix $I_{d+1} + \Lambda \, \Kb_n$ can be easily
verified to be
\begin{align*}
   (I_{d+1} + \Lambda \, \mathbf{K}_n)^{-1} \Lambda = & \,
\frac{M}{[1+M({A_n} - {B_n})][1+M\,{A_n} + d\,M\,{B_n}]} \\
&  \times\left[(1+M\,{A_n} +  d\,M\,{B_n})I_{d+1} - M\,{B_n}
   \begin{pmatrix} 1 & \cdots & 1\\\vdots &  & \vdots\\
1 & \cdots & 1\end{pmatrix}\right].
\end{align*}
As a result, we can now use Theorem \ref{main-thm} to derive an explicit orthogonal basis
for the inner product \eqref{simplex-ip}, which is given by
\begin{align*}
\mathbb{Q}_n(x)  = \, & \mathbb{P}_n(x) +
\frac{M}{1+M(A_{n-1} - B_{n-1})} \sum_{i=1}^{d+1}
\mathbb{P}_n(e_i)\, K_{n-1}(W_{\kappa};x,e_i)\\
 & -  \frac{M^2 \,B_{n-1}}{[1+M(A_{n-1} - B_{n-1})][1 +M\,A_{n-1} + d\,M\,{B_{n-1}}]}\times\\
 & \quad  \times\sum_{i=1}^{d+1} \mathbb{P}_n(e_i) \,
         \sum_{i=1}^{d+1} K_{n-1}(W_{\kappa};x,e_i),
\end{align*}
where $\{\mathbb{P}_n\}_{n\ge0}$ denotes the orthonormal polynomial system on
the simplex $T^d$ given by \eqref{simplex-base}. Furthermore, by Theorem \ref{thm.3.3},
the reproducing kernel $K_n(d\nu;x,y)$ for $\Pi_n^d$ under the inner product
\eqref{simplex-ip} is given by
\begin{align}
K_n(d\nu; x,y)  = & K_n(W_{\kappa};x,y) +
   \frac{M}{1+M({A_n} - {B_n})} \sum_{i=1}^{d+1}
   K_n(W_{\kappa};x,e_i)\,K_n(W_{\kappa};y,e_i)\nonumber \label{kernel-simplex}\\
& - \frac{M^2 \,{B_n}}{[1+M({A_n} - {B_n})][1 +M\,{A_n} + d\,M\,{B_n}]}\times\\
& \quad \times\sum_{i=1}^{d+1} K_n(W_{\kappa};x,e_i)
\sum_{i=1}^{d+1} K_n(W_{\kappa};y,e_i). \nonumber
\end{align}

The explicit formula of the kernel allows us to derive a sharp estimate for the
kernel $K_n(d\nu;x,y)$ from those for $K_n(W_\k;x,y)$ and for the Jacobi polynomials.
In the case of one variable ($d=1$), such an estimate has been carried out in \cite{GPRV}.
We shall give one result on the strong asymptotic of the Christoffel function with respect
to $d\nu$ on the simplex $T^d$. For this purpose, we will need the following estimate of
the Jacobi polynomials (\cite[(7.32.5) and (4.1.3)]{Sz}):

\begin{lemma} \label{lem:3.2}
For an arbitrary real number $\alpha$ and $t \in [0,1]$,
\begin{equation} \label{Est-Jacobi}
|P_n^{(\alpha,\beta)} (t)| \le c n^{-1/2} (1-t+n^{-2})^{-(\alpha+1/2)/2}.
\end{equation}
The estimate on $[-1,0]$ follows from the fact that $P_n^{(\alpha,\beta)}
(t) = (-1)^nP_n^{(\beta, \alpha)} (-t)$.
\end{lemma}

Note that \eqref{Est-Jacobi} shows that $|P_n^{(\a,\b)}| \le c n^{-1/2}$ uniformly inside
a compact subset of $(-1,1)$. We derive the asymptotic for the difference
$K_n(d\nu;x,x) - K_n(W_{\kappa};x,x)$.

\begin{thm} \label{thm:K-K}
For $x$ in $T^d$,
\begin{align} \label{Kn-Kn}
 K_n(d\nu; x,x)- & {K}_n(W_{\kappa};x,x) =
  \frac{1}{2^{d+1}} \frac{\Gamma(\lambda-\varsigma + 1/2)
    \Gamma(\varsigma + 1/2)}{\Gamma(\lambda)} \\
    & \times  \sum_{i=1}^{d+1}
 \left[P^{(\lambda-\varsigma-1/2,\varsigma-1/2)}_{n}(2\, x_i-1)\right]^2 \left (1+ \CO(n^{-1})\right). \notag
\end{align}
In particular, for $x$ in the interior of $T^d$,
$$
    \lim_{n\to \infty} \left[K_n(d\nu; x,x)-   {K}_n(W_{\kappa};x,x) \right] = 0,
$$
and the convergence is uniform in any compact set in the interior of $T^d$.
\end{thm}

\begin{proof}
From \eqref{kernel-simplex} we deduce
\begin{align*}
 K_n(d\nu; x,x)- & {K}_n(W_{\kappa};x,x) =
\frac{M}{1+M({A_n} - {B_n})} \sum_{i=1}^{d+1} {K}_n(W_{\kappa};x,e_i)^2 \\
&- \frac{M^2 \,{B_n}}{[1+M({A_n} - {B_n})][1  +M\,{A_n} +
d\,M\,{B_n}]}\times\left(\sum_{i=1}^{d+1}K_n(W_{\kappa};x,e_i)\right)^2\\
=\, & \frac{M C_n^2}{1+M({A_n} - {B_n})} \sum_{i=1}^{d+1}
 \left[P^{(\lambda-\varsigma-1/2,\varsigma-1/2)}_{n}(2\, x_i-1)\right]^2 \\
 & - \frac{M^2 C_n^2 B_n}{[1+M({A_n} - {B_n})][1 + M\,{A_n} + d\,M\,{B_n}]}\times\\
& \quad \times\left(\sum_{i=1}^{d+1}
    P^{(\lambda-\varsigma-1/2,\varsigma-1/2)}_{n}(2\, x_i-1)\right)^2,
\end{align*}
where
$$
C_n = \frac{1}{2^{d+1}} \, \frac{(\lambda)_n}{(\varsigma+1/2)_n}.
$$
By the Stirling formula for the Gamma function (see \cite[(6.1.39), p. 257]{AS}), we have
$$
      \frac{ \Gamma(n+a) }{\Gamma(n+1)} = n^{a-1} (1+ \CO(n^{-1}))
$$
as $n\to\infty$. Consequently, it is easy to see that the following limit relations hold:
\begin{align*}
  \frac{M C_n^2}{1+M({A_n} - {B_n})}  = \, &
  \frac{1}{2^{d+1}} \frac{\Gamma(\lambda-\varsigma + 1/2)
  \Gamma(\varsigma + 1/2)}{\Gamma(\lambda)} \left (1+ \CO(n^{-1})\right), \\
  \frac{M B_n}{1+ M\,{A_n} + d\,M\,{B_n}} = \, &  (-1)^n \frac{\Gamma(\lambda-\varsigma + 1/2)}
     {\Gamma(\varsigma + 1/2)} n^{-(\l - 2 \varsigma -1)}  \left (1+ \CO(n^{-1})\right).
\end{align*}
Since $\l - 2 \varsigma -1 = (d-1)(\varsigma +1/2) > 0$ for $d \ge 2$ and
$$
\left(\sum_{i=1}^{d+1}
    P^{(\lambda-\varsigma-1/2,\varsigma-1/2)}_{n}(2\, x_i-1)\right)^2 \le (d+1)
       \sum_{i=1}^{d+1} \left[P^{(\lambda-\varsigma-1/2,\varsigma-1/2)}_{n}(2\, x_i-1)\right]^2
$$
by the Cauchy--Schwarz inequality, it follows readily that
\begin{equation*}
   K_n(d\nu; x,x) - K_n(W_{\kappa};x,x) = c_{\l,\varsigma} \sum_{i=1}^{d+1}
 \left[P^{(\lambda-\varsigma-1/2,\varsigma-1/2)}_{n}(2\, x_i-1)\right]^2 \left (1+ \CO(n^{-1})\right),
\end{equation*}
where $c_{\l,\varsigma}$ is the constant
$$
   c_{\l,\varsigma}=  \frac{1}{2^{d+1}} \frac{\Gamma(\lambda-\varsigma + 1/2)
    \Gamma(\varsigma + 1/2)}{\Gamma(\lambda)}.
$$
This is \eqref{Kn-Kn}. If $x$ is in the interior of $T^d$, then
$|P_n^{(\lambda-\varsigma-1/2,\varsigma-1/2)}(2\, x_i-1)| \le c n^{-1/2}$, so that
$K_n(d\nu; x,x) - K_n(W_{\kappa};x,x)$ goes to zero as $n \to \infty$.
\end{proof}

The asymptotic of the Christoffel function for $W_\k$ was studied in \cite{X99}, where
most of the results were for convergence in the interior of $T^d$. Such results carry over
to $K_n(d\nu; x,y)$ by Theorem \ref{thm:K-K}. In one particular case, $\kappa =0$, the
convergence holds for all $T^d$ as given in \cite[Theorem 2.3]{X99}:
\begin{equation}\label{limitK}
   \lim_{n \to \infty} \frac{1}{\binom{n+d}{n}} K_n(W_0;x,x) = 2^{d- k}, \qquad x \in T_k^d, \quad
    0 \le k \le d,
\end{equation}
where $T_k^d$ denotes the {\it $k$--dimensional face} of $T^d$, which contains elements
of $T^d$ for which exactly $d-k$ inequalities in $T^d=\{x: x_1\ge 0, \ldots, x_{d+1} \ge 0\}$
becomes equalities. In particular, $T_d^d$ (when none of the inequalities become equality)
is the interior of $T^d$, and $0$--dimensional face $T_0^d$ is the set of the vertices.
Setting $\kappa =0$, so that $\varsigma =0$ and $\l = (d+1)/2$, we see that
$$
  P_n^{(\lambda-\varsigma-1/2,\varsigma-1/2)}(2\, x_i-1) = P_n^{(d/2, -1/2)}(2\, x_i-1)
     = (-1)^n P_n^{(-1/2, d/2)}(1-2\, x_i),
$$
which is bounded by $c n^{-1/2}$ whenever $1- x_i \ge \varepsilon  >0$. It follows then that
$$
  \lim_{n\to \infty} \frac{1}{\binom{n+d}{n}}
        \sum_{i=1}^{d+1} \left[P_n^{(d/2, -1/2)}(2\, x_i-1) \right]^2
        = \begin{cases} 0 & x \in T_k^d, \,\, k >0 \\ 1 & x \in T_0^d, \end{cases}
$$
upon using the fact that $P_n^{(a,b)}(1) = \binom{n+a}{n}$.
By \eqref{Kn-Kn}, we then end up with the following corollary.

\begin{cor}
For $\k = 0$,
$$
   \lim_{n \to \infty} \frac{1}{\binom{n+d}{n}} K_n(d\nu;x,x) = \begin{cases}
      2^{d- k}, & x \in T_k^d,  \,\,    k > 0, \\
      2^d +  E_d, & x \in T_0^d, \end{cases}
$$
where $E_d = c_{(d+1)/2,0} =  \Gamma(d/2+1) \sqrt{\pi} /(\Gamma(d+1/2) 2^{d+1})$.
\end{cor}

Comparing to \eqref{limitK}, the result shows the impact of the additional mass points at
the vertices. More generally, if $\kappa_i = \varsigma > 0$, then \eqref{Kn-Kn} shows that
\begin{align*}
 K_n(d\nu;x,x) - K_n(W_\k;x,x) = & c_{\l,\varsigma} \left (1 + \CO(n^{-1}) \right) \\
    & \times \begin{cases} \binom{n + \l - \varsigma  -1/2}{n}, & x \in T_0^d, \\
     2^k \binom{n + \varsigma  -1/2}{n},  & x \in T_k^d, \,\,1\le k \le d-1, \end{cases}
\end{align*}
since, for $x\in T^d$, $x_i=1$ only when $x = e_i$. In particular, we see that
$$
   \lim_{n \to \infty} \frac{1}{\binom{n+d}{n}} \left[K_n(d\nu;x,x) - K_n(W_\k;x,x) \right] = 0,
     \quad x \in T_k^d, \quad 1 \le k \le d,
$$
if $d > 2 \varsigma -1$, whereas this limit is unbounded when $x \in T_0^d$.

\end{document}